\documentclass{amsart}
\usepackage{amssymb}
\usepackage{mathrsfs}
\usepackage{graphicx}
\usepackage{bbm}
\usepackage{amsmath}
\usepackage{mathtools}
\usepackage{graphics}
\usepackage{color}
\usepackage{subfigure}
\usepackage{float}
\usepackage{multirow}
\usepackage{longtable}
\usepackage{caption}
\restylefloat{table}
\usepackage{titletoc}
\usepackage{rotating}

\usepackage{amsfonts} 

\usepackage{wrapfig}
\usepackage{amsthm}  

\usepackage[toc,page,titletoc,title]{appendix}

\usepackage[english]{babel}
\usepackage[a4paper, total={6in, 8in}]{geometry}

\theoremstyle{definition}

\newtheorem{theorem}{Theorem}
\newtheorem{corollary}[theorem]{Corollary}

\newtheorem{example}[theorem]{Example}

\newtheorem{lemma}[theorem]{Lemma}

\newtheorem{remark}[theorem]{Remark}

\begin{document}

\title[A Switching Levy Model]{Characteristic function and Esscher transform of a  switching  Levy model for the temperature dynamic }%
\author{Rofeide Jabbari and Pablo Olivares}%
\address{}%
\email{polivare@fiu.edu}%

\thanks{}%

\begin{abstract}
In this paper we  extend the models in \cite{rstemp1,swttemp} for the dynamic of the temperatures by considering random switching between Levy noises instead of Brownian motions, with a mean-reverting movement towards a seasonal periodic function. The use of Levy noises allows for jumps, capturing, together with the regime changes, sudden and relatively persistent oscillations in the weather. An approximated close-form expression for the characteristic function of the temperature process under an Esscher transform is given.
\end{abstract}
\maketitle
\section{Introduction}
Switching models for temperatures have been recently proposed in \cite{rstemp1, swttemp}. In such models the temperature  evolves between mean-reverting stochastic differential equations whose background noises switch at random times between Brownian motions with different volatilities. We extend these models by considering random switching between Levy noises with  mean-reverting movement towards a seasonal periodic function. The use of Levy noises and regime-switching models allow  random jumps in the underlying temperature process, capturing sudden and relatively persistent oscillations in temperatures and   weather in general. \\
 With a view on the pricing weather of derivatives we find the characteristic function of integrals with respect to Levy switching processes by conditioning on the number of switches over a given time interval and investigate the choice of an equivalent martingale measure(EMM) to create a risk-neutral setting with the  use of an Esscher transform. See \cite{EsscherTransform} for the use of Esscher transform. We implement the findings in  time-changed Levy models driven by inverse Gaussian and Gamma subordinators.\\
 In the context of discrete-time regime switching models   have been successfully considered since the pioneer work of \cite{switchfirst}. Continuous-time switching Levy models have been recently introduced in  \cite{conradferoli} in connection with the modeling of oil prices and in \cite{alexatousa} to price spread derivative contracts.\\
 The modeling of the temperature dynamic using Levy noises without switching has been considered in a series of papers. See for example \cite{switchfirst,alfonsi2023stochasticvolatilitymodelvaluation,benth,olivillalevytemp}.\\
The organization of the paper is the following:\\
In section 2 we present the switching Levy model for temperature and compute the characteristic function of the temperatures under the historic measure. In section 3 we find the characteristic function of the temperatures under the measure generate by an Esscher transform, as long as the selection of the parameter in the Esscher transform to obtain an EMM risk-neutral measure. In section 4 we discuss the numerical implementation in the cases of time-changed Levy models with inverse Gaussian and Gamma subordinators while in section 5 we conclude.
\section{A switching Levy model for temperature}
 Let  $(\Omega ,\mathcal{A}, (\mathcal{F}_{t})_{t \geq 0}, P)$ be a filtered probability space verifying the usual conditions. For a stochastic process $(X_t)_{t \geq 0}$ defined on the space filtered space above the functions $\varphi_{X_{t}}$  and $l_{X_t}(u)= \frac{1}{t} \log \varphi_{X_{t}}(-iu)$ represent its characteristic function and the log-cumulant generating function respectively. Notice that when the process has stationary and independent increments the later does not depend on $t$.\\
  The  class $\sigma_{a,b}(\tau),\; a<b \in \mathbb{N}$ is the $\sigma$-algebra generated by the random variables $\{ \tau_a, \tau_{a+1}, \ldots, \tau_b \}$. In particular, we write $\sigma_{b}(\tau):=\sigma_{0,b}(\tau)$.\\
 The discounted process $(\tilde{X}_t)_{t \geq 0}$ is defined as $\tilde{X}_t=e^{-rt} X_t$, where $r > 0$ is a constant interest rate.\\
 Furthermore,  $M(a,b,z)$ denotes the Kummer's confluent hypergeometric function:
\begin{equation*}
  M(a,b,z)=\frac{\Gamma(b)}{\Gamma(a)\Gamma(b-a)}\int_0^1 e^{z u}u^{a-1}(1-u)^{b-a-1}\;du
\end{equation*}
The confluent hypergeometric function can be expanded as:
\begin{equation*}
  M(a,b,z)=\sum_{l=0}^{+\infty} \frac{a^{(l)}}{b^{(l)}l!}z^l
\end{equation*}
See \cite{hypergeom} for definition and properties.\\
Also, $\gamma(a,z)$ is the incomplete Gamma function:
\begin{equation*}
  \gamma(a,z)= \int_0^z x^{a-1} e^{-x}\;dx,\;a \geq 0
\end{equation*}
Let $(T_t)_{t \geq 0}$ be a stochastic process where $T_t$ is the temperature in certain location at time $t$. Temperatures  switch between two regimes. These changes are driven by a stationary Markov process $\{r_t \}_{t \geq 0}$ with values in the space $E_r=\{1,2 \}$.\\
Moreover,  under regime $j$ the temperature process $(T^j_t)_{t \geq 0}$ verifies the mean-reverting stochastic differential equation:
\begin{equation}\label{eq:tempdyna}
    dT^j_t= ds^j_t+\alpha_j(s^j_t-T^j_t)dt+  \sigma_j dV^j_t
\end{equation}
where  the temperature reverts to the deterministic seasonal process  $(s_t)_{t \geq 0}$ given by:
\begin{eqnarray} \label{eq:seass2}
s^j_t&=&  \beta^j_0 + \beta^j_1 t+ \beta^j_2 \sin \left(\frac{2 \pi}{365} t \right)+\beta^j_3 \cos \left(\frac{2 \pi}{365}t \right)
\end{eqnarray}
The parameter $\alpha_j >0$ is the mean-reversion rates under regime $j$, while $\sigma_j$ is the  volatility or standard deviation parameter also under  regimes $j$.  For simplicity we consider the case of constant coefficients in the seasonal component i.e $\beta^j_k=\beta_k, \;k=0,1,2,3$, the mean-reverting rate, $\alpha_j=\alpha$ and  the volatility $\sigma_j=\sigma$. Hence, the switching is reflected only in the Levy noise.\\
Notice that  switching between Levy processes may lead to a change in the volatilities in the volatility through a change in  their parameters.\\
The general case adds some technicalities in the solution of equation (\ref{eq:tempdyna}) but does not change essentially the approach.\\
 To this end consider two switching subordinator  processes $(R^j_t)_{t \geq 0}, j=1,2$, where $R^j$ is  corresponding to regime $j$, and  time-changed processes $(V^l_t)_{t \geq 0}$  defined as:
\begin{eqnarray}\label{eq:tchlevy2}
  V^j_t&=& B_{R^j_t}+ \mu^j_1 R^j_t,\;j=1,2
\end{eqnarray}
Here  $\mu^j_1   \in \mathbb{R},j=1,2.$  are   parameters in the model and $(B_t)_{t \geq 0}$ is a standard Brownian motion. It is assumed that both subordinators  have finite moments of some convenient order depending of the specific process chosen.\\
On the other hand, the dynamic of the regime-switching Markov process is given by the following transition probabilities:
 \begin{eqnarray*}
  P \{r_{t+h}=j/r_t=k \} &=&  \lambda_{jk}h+o(h),\; j \neq k\\
   P \{r_{t+h}=j/r_t=j \} &=&  -\lambda_{jj}h+o(h)
\end{eqnarray*}
Without loss of generality we assume the initial conditions $r_0=1$, i.e. the process starts  under regime $1$. The analysis of a process starting under regime 2 or at a random initial value is similar.\\
Let $N_t$  be  the number of transitions between regimes in the interval $[0,t)$ and  $(\tau_n)_{n \in \mathbb{N}}$ be an increasing  sequence of stopping times representing the instants of regime changes. Notice that $\tau_0=0$ and $\tau_{N_t} \leq t$.\\
 Set also $v_j=\tau_j-\tau_{j-1},\; j \in \mathbb{N}$ as the times between two consecutive changes  of regime.
 \begin{remark}
  Because of the Markovian nature of the switching process the random variables $v_j$ are independent such that  $v_j \sim exp(\lambda_{12})$ if the transition occurs from regime 1 to regime 2 and $v_j \sim exp(\lambda_{21})$ otherwise, where $\lambda_{12}>0, \lambda_{21}>0$.
  \end{remark}
The solution of equation (\ref{eq:tempdyna}) is given in the next elementary result. It follows  from Ito lemma.
\begin{lemma}\label{le:ito}
The solution of equation (\ref{eq:tempdyna}) is:
\begin{eqnarray} \label{eq:timechansol}
  T_t&=&C_1(t, \alpha)+  \sigma e^{-\alpha t} W_t
  \end{eqnarray}
with
\begin{eqnarray}\label{eq:swithdec}
  W_t &=& \sum_{j=1}^{N_t} \int_{\tau_{j-1}}^{\tau_{j}}  e^{\alpha u}dV^{\nu(j)} _u+ \int_{\tau_{N_t}}^{t}  e^{\alpha u}dV^{\nu(N_t+1)} _u
\end{eqnarray}
where $\nu(j)=\frac{3+(-1)^j}{2}$ and $C_1(t, \alpha) = s_t+e^{-\alpha t}(T_0-s_0)$
\end{lemma}
A  result about the probability distribution of the instants and the number of regime changes in $[0,t)$ is given in the next two lemmas. To this end we introduce the following quantities:
\begin{eqnarray*}
 D(m,n,l)&=& \frac{\lambda^{ m}_{12}}{\lambda^{ n}_{21} \Gamma(l)}, \; l \geq 1 \\
 D(m,n,0)&=& 0,\; C(m,n,0)= 1 \\
  C(m,n,l) &=& \frac{m^{(l)}(1-\frac{\lambda_{12}}{\lambda_{21}})^l}{n^{(l)}l!}=\frac{\binom{m+l-1}{l}(1-\frac{\lambda_{12}}{\lambda_{21}})^l}{\binom{n+l-1}{l}l!},\; l>0
\end{eqnarray*}
where $a^{(l)}$ is defined as the Pochhammer symbol or raising factorial,
\begin{equation*}
  a^{(0)}=1, a^{(l)}=a(a+1) \ldots (a+l-1)
\end{equation*}
 We have, for $k \geq 1$:
\begin{lemma}\label{lemma:tausdistrib}
Let $\tau_l$ be the time of the l-th regime change, $f_{\tau_l}$ and $F_{\tau_l}$ its respective  probability density function (p.d.f.) and cumulative distribution function(c.d.f.). Then:
\begin{eqnarray*}
 F_{\tau_{2k}}(t)  &=& D(k,k,2k)  \sum_{l=0}^{+\infty}  C(k,2k,l)  \gamma(2k+l, \lambda_{21} t) \\
  F_{\tau_{2k+1}}(t)  &=&  D(k+1,k+1,2k+1) \sum_{l=0}^{+\infty} C(k+1,2k+1,l)  \gamma(2k+l+1, \lambda_{21} t)
\end{eqnarray*}
Moreover, let $N_t$  be  the number of transitions between regimes  in the interval $[0,t)$. Then, for $k \in \mathbb{N}$:
\begin{eqnarray} \label{eq:poisson1}
p_k(t):=P(N_t=k)  &=& F_{\tau_{k}}(t)-F_{\tau_{k+1}}(t)
\end{eqnarray}
\end{lemma}
\begin{proof}
First, define the independent random variables:
\begin{eqnarray*}
 X_1 &=& \sum_{l=1}^k v_{2l-1},\;  X_2 = \sum_{l=1}^k v_{2l}
\end{eqnarray*}
Notice that $\tau_{2k}=X_1+X_2$, where $X_1 \sim Erlang(k,\lambda_{12})$ and $X_2 \sim Erlang(k,\lambda_{21})$. \\
The probability distribution of the sum of two independent Erlang random variables with different shape and rate parameters, or equivalently the sum of independent exponential random variables with different rates,  has been found in \cite{paganis} or in \cite{levyedmond} for example. Adapting their results, specifically  Corollary 6.3 in \cite{levyedmond}, to our case we have:
\begin{eqnarray*}
  f_{\tau_{2k}}(x)  &=&D(k,-k,2k)  M(k,2k,(\lambda_{21}-\lambda_{12})x)x^{2k-1}e^{-\lambda_{21} x},\;x \geq 0,\; \lambda_{21}> \lambda_{12}
\end{eqnarray*}
Therefore for $t \geq 0$:
\begin{eqnarray*}
 F_{\tau_{2k}}(t)  &=& D(k,-k,2k)   \int_0^t x^{2k-1}e^{-\lambda_{21} x}M(k,2k,(\lambda_{21}-\lambda_{12})x)\;dx\\
 &=& D(k,-k,2k)  \sum_{l=0}^{+\infty} C(k,2k,l)    \int_0^t x^{2k+l-1}e^{-\lambda_{21} x}\;dx \\
  &=& D(k,k,2k)  \sum_{l=0}^{+\infty}  C(k,2k,l)  \gamma(2k+l, \lambda_{21} t)
\end{eqnarray*}
Similarly $\tau_{2k+1}=Y_1+Y_2$, where $Y_1 \sim Erlang(k+1,\lambda_{12})$ and $Y_2 \sim Erlang(k,\lambda_{21})$ and independent random variables. Then:
\begin{eqnarray*}
f_{\tau_{2k+1}}(x) &=& D(k+1,-k,2k+1)  M(k+1,2k+1,(\lambda_{21}-\lambda_{12})x)  x^{2k}e^{-\lambda_{21} x}
\end{eqnarray*}
which leads to:
\begin{eqnarray*}
 F_{\tau_{2k+1}}(t)  &=& D(k+1,-(k+1),2k+1) \int_0^t x^{2k}e^{-\lambda_{21} x}M(k+1,2k+1,(\lambda_{21}-\lambda_{12})x)\;dx\\
  &=& D(k+1,k,2k+1) \sum_{l=0}^{+\infty} C(k+1,2k+1,l)  \gamma(2k+l+1,\lambda_{21} t)
 \end{eqnarray*}
 Regarding the second part of the lemma we have:
\begin{eqnarray*}
p_k(t)  &=&  P( \tau_{k} <t, \tau_{k+1}>t)= P( \tau_{k} <t)-P(\tau_{k+1}<t)=F_{\tau_{k}}(t)-F_{\tau_{k+1}}(t)
\end{eqnarray*}
\end{proof}
Next, we compute the characteristic function of some integrals of the background noise process given by the switching Levy process $(V_t)_{t \geq 0}$. To this end we will use a well-known result  about functional  of a Levy process $(\xi_t)_{t \geq 0}$ with $\xi_0=0$ and a measurable function $f$, see for example \cite{eberleinraible}. Namely:
\begin{equation}\label{eq:funclevy}
  E [ exp( i \int_0^t f(s)\;d X_s) ]=exp(\int_0^t l_{X}(-i f(s))\;ds)
\end{equation}
where  $(X_t)_{t \geq 0}$ is a Levy process with $X_0=0$ and $f$ is a measurable function.\\
The main result in the section is given below.
\begin{theorem}\label{teo:chfintpro}
 Let $(\xi_t)_{t \geq 0}$ be a two-regime switching Levy process starting at regime one, with  log-cumulant generating function $l_{\xi^l}$ when the process is at regime $l=1,2$. Then,
 \begin{eqnarray} \nonumber
  E [ exp( i \int_0^t f(s)\;d \xi_s) ]&=&  \left( \sum_{k=0}^{+\infty}  M_1(k) [ p_{2k}(t)+ \lambda_{12} J_1  p_{2k+1}(t) ] \right)\\ \nonumber
   && \left( \sum_{k=0}^{+\infty}   J_3(t,k)   p_{2k}(t) +  \sum_{k=0}^{+\infty}   J_4(t,k) p_{2k+1}(t) \right) \\ \label{eq:funclevy3}
  &&
\end{eqnarray}
where:
\begin{eqnarray*}
 J_1&=& \int_0^{+\infty} e^{I_{1}(x)}e^{-\lambda_{12}x}\; dx, \;\;J_2= \int_0^{+\infty} e^{I_{2}(x)}e^{-\lambda_{21}x}\; dx \\
 J_3(t,k)&=& D(k, -k, 2k) \sum_{l=0}^{+\infty}C(k,2k,l) G_1(t,2k+l)\lambda^l_{21} \\
 J_4(t,k) &=& D(k+1, -k, 2k+1) \sum_{l=0}^{+\infty}C(k+1,2k+1,l) G_2(t,2k+l+1) \lambda^l_{21} \\
G_j( t,m) &=& \int_0^{t} e^{I_j(t-z)}  z^{m-1}e^{-\lambda_{21}  z} \;dz,\; j=1,2 \\
 M_1(k) &=& \lambda^k_{12} \lambda^k_{21} J^k_1 J^k_2
\end{eqnarray*}
and
\begin{equation*}
  I_j(x)=\int_0^x l_{\xi^{j}}(-if(s))\;  ds, \;\;j=1,2
\end{equation*}
Expressions for $ p_{2k}(t)$ and  $ p_{2k+1}(t)$ are given in the  lemma \ref{lemma:tausdistrib}.
\end{theorem}
\begin{proof}
 Notice that between times of regime changes  $(\xi_t)_{t \geq 0}$ is a Levy process. Hence, equation (\ref{eq:funclevy}) applies after conditioning on the random variables $\tau_{j-1}$ and $\tau_j$. \\
Moreover, conditioning on $N_t=2k$ and $\sigma_{2k}(\tau)$ and from the independence and the stationarity of the increments of a Levy process we have:
\begin{eqnarray*}
&&  E[\exp( i \sum_{j=1}^{N_t} \int_{\tau_{j-1}}^{\tau_{j}} f(s)\;  d\xi^{\nu(j)}_s /  N_t=2k, \sigma_{2k}(\tau))]= E[\exp( i \sum_{j=1}^{2k} \int_{\tau_{j-1}}^{\tau_{j}} f(s)\;  d\xi^{\nu(j)}_s / \sigma_{2k}(\tau))]\\
&=&  \prod_{j=1}^{2k} E[\exp( i \int_{\tau_{j-1}}^{\tau_{j}} f(s)\;  d\xi^{\nu(j)}_s / \sigma_{2k}(\tau))]= \prod_{j=1}^{2k} E[ \exp( i \int_{0}^{v_{j}} f(s)\;  d\xi^{\nu(j)}_s)] \\
&=& \exp[ \sum_{j=1}^{2k} \int_{0}^{v_{j}} l_{\xi^{\nu(j)}}(-if(s))\;  ds]= \exp[ \sum_{j=1}^{2k} I_{\nu(j)}(v_{j})]
\end{eqnarray*}
and
\begin{eqnarray*}
&& E[\exp( i \int_{\tau_{N_t}}^{t}  f(s) d\xi^{1} _s / N_t=2k, \sigma_{2k}(\tau))]=E[\exp( i \int_{\tau_{2k}}^{t} f(s)\;  d \xi^{1}_s /  \sigma_{2k}(\tau) )]\\
& = & E[ \exp( i \int_{0}^{t-\tau_{2k}} f(s)\;  d\xi^{1}_s) /  \sigma_{2k}(\tau)]\\
&=&  \exp[  \int_{0}^{t-\tau_{2k}} l_{\xi^{1}}(-if(s))\;  ds ]=\exp[ I_{1}(t-\tau_{2k})]
\end{eqnarray*}
 By similar analysis, conditioning   on $[N_t=2k+1]$, we have:
 \begin{eqnarray*}
&&  E[\exp( i \sum_{j=1}^{N_t} \int_{\tau_{j-1}}^{\tau_{j}} f(s)\;  d\xi^{\nu(j)}_s) /  N_t=2k+1, \sigma_{2k+1}(\tau)]= E[\exp( i \sum_{j=1}^{2k+1} \int_{\tau_{j-1}}^{\tau_{j}} f(s)\;  d\xi^{\nu(j)}_s / \sigma_{2k+1}(\tau))]\\
&=&  \prod_{j=1}^{2k+1} E[\exp( i \int_{\tau_{j-1}}^{\tau_{j}} f(s)\;  d\xi^{\nu(j)}_s / \sigma_{2k+1}(\tau))]=  \prod_{j=1}^{2k+1} E[ \exp( i \int_{0}^{v_{j}} f(s)\;  d\xi^{\nu(j)}_s )] \\
&=& \exp[ \sum_{j=1}^{2k+1} \int_{0}^{v_{j}} l_{\xi^{\nu(j)}}(-if(s))\;  ds]= \exp[ \sum_{j=1}^{2k+1} I_{\nu(j)}(v_{j})]
\end{eqnarray*}
and
\begin{eqnarray*}
&& E[\exp( i \int_{\tau_{N_t}}^{t}  f(s) d\xi^{1} _s) / N_t=2k+1, \sigma_{2k+1}(\tau)]=E[\exp( i \int_{\tau_{2k+1}}^{t} f(s)\;  d \xi^{2}_s )/  \sigma_{2k+1}(\tau) ]\\
& = & E[ \exp( i \int_{0}^{t-\tau_{2k+1}} f(s)\;  d\xi^{2}_s)]=\exp[  \int_{0}^{t-\tau_{2k+1}} l_{\xi^{2}}(-if(s))\;  ds ] \\
&=& \exp[ I_{2}(t-\tau_{2k+1})]
\end{eqnarray*}
If there are $2k$ changes of regime on the interval $[0,t)$ there will be $k$ subintervals where the process is at regime one and another $k$ where the process is at regime two. The remaining time on $[0,t)$, i.e. during the interval $[t_{N_t},t)$,  the process is at regime one.\\
 In a  similar analysis, if there are $2k+1 $ changes of regime on the interval $[0,t)$ there will be $k+1$ subintervals where the process is at regime one and another $k$ where the process is at regime two.   The process remains in regime two during $[\tau_{N_t},t)$.\\
On the other hand:
\begin{equation*}
   F_{t-\tau_{N_t}}(x)= P(\tau_{N_t} \geq t-x)=1- F_{\tau_{N_t}}(t-x), \;x<t
\end{equation*}
Therefore, conditionally on the event $[N_t=2k]$  we have for $ 0 \leq x \leq t$ that:
\begin{eqnarray*}
 f_{t-\tau_{2k}}(x)&=& f_{\tau_{2k}}(t-x)=D(k,-k,2k) (t-x)^{2k-1}e^{-\lambda_{21}  (t-x)}M(k,2k,(\lambda_{21}-\lambda_{12}) (t-x))
\end{eqnarray*}
Similary,  conditionally on $[N_t=2k+1]$:
\begin{eqnarray*}
 f_{t-\tau_{2k+1}}(x)&=& D(k+1,-k,2k+1)  (t-x)^{2k}e^{-\lambda_{21}  (t-x)}M(k+1,2k+1,(\lambda_{21}-\lambda_{12}) (t-x))
\end{eqnarray*}
 and zero otherwise.\\
Furthermore:
\begin{eqnarray*}
 && E[\exp( i \sum_{j=1}^{N_t} \int_{\tau_{j-1}}^{\tau_{j}} f(s)\;  d\xi^{\nu(j)}_s)/ \sigma_{N_t}(\tau)] \\
 &=& \sum_{k=0}^{+\infty} E[\exp( i \sum_{j=1}^{N_t} \int_{\tau_{j-1}}^{\tau_{j}} f(s)\;  d\xi^{\nu(j)}_s) /N_t=2k,\sigma_{2k}(\tau)]p_{2k}(t)\\
 &+& \sum_{k=0}^{+\infty} E[\exp( i \sum_{j=1}^{N_t} \int_{\tau_{j-1}}^{\tau_{j}} f(s)\;  d\xi^{\nu(j)}_s) /N_t=2k+1,\sigma_{2k+1}(\tau)]p_{2k+1}(t)\\
 &=& \sum_{k=0}^{+\infty} \exp[ \sum_{j=1}^{2k} I_{\nu(j)}(v_{j})]p_{2k}(t)+ \sum_{k=0}^{+\infty} \exp[ \sum_{j=1}^{2k+1} I_{\nu(j)}(v_{j})]p_{2k+1}(t)
\end{eqnarray*}
and
\begin{eqnarray*}
 && E[\exp( i \int_{\tau_{N_t}}^{t}  f(s) d\xi^{1} _s)/ \sigma_{N_t}(\tau)] \\
 &=& \sum_{k=0}^{+\infty} E[\exp( i \int_{\tau_{N_t}}^{t}  f(s) d\xi^{1} _s) / N_t=2k, \sigma_{2k}(\tau)]p_{2k}(t)\\
 &+& \sum_{k=0}^{+\infty} E[\exp( i \int_{\tau_{N_t}}^{t}  f(s) d\xi^{2} _s) / N_t=2k+1, \sigma_{2k+1}(\tau)]p_{2k+1}(t)\\
 &=& \sum_{k=0}^{+\infty} \exp[ I_{1}(t-\tau_{2k})]p_{2k}(t)+ \sum_{k=0}^{+\infty} \exp[I_{2}(t-\tau_{2k+1})]p_{2k+1}(t)
\end{eqnarray*}
Moreover:
\begin{eqnarray*}
 && E[\exp( i \sum_{j=1}^{N_t} \int_{\tau_{j-1}}^{\tau_{j}} f(s)\;  d\xi^{\nu(j)}_s)]=E[E[\exp( i \sum_{j=1}^{N_t} \int_{\tau_{j-1}}^{\tau_{j}} f(s)\;  d\xi^{\nu(j)}_s)/ \sigma_{N_t}(\tau)]] \\
 &=& \sum_{k=0}^{+\infty} \prod_{j=1}^{2k}E[\exp(  I_{\nu(j)}(v_{j}))]p_{2k}(t)+ \sum_{k=0}^{+\infty} \prod_{j=1}^{2k+1} E[\exp(  I_{\nu(j)}(v_{j}))]p_{2k+1}(t)\\
 &=& \sum_{k=0}^{+\infty} (E[\exp(  I_{1}(v))])^k (E[\exp(  I_{2}(v^*))])^kp_{2k}(t)\\
 &+& \sum_{k=0}^{+\infty} (E[\exp(  I_{1}(v))])^{k+1} (E[\exp(  I_{2}(v^*))])^kp_{2k+1}(t)
 \end{eqnarray*}
where $v \sim exp(\lambda_{12})$ is a replica of the time between changes from regime one to regime two and $v* \sim \exp(\lambda_{21})$ is a replica of the time between changes from regime two to regime one.\\
Also,
\begin{eqnarray*}
 && E[\exp( i \int_{\tau_{N_t}}^{t}  f(s) d\xi^{\nu(N_t+1)} _s)]=E[E[\exp( i \int_{\tau_{N_t}}^{t}  f(s) d\xi^{\nu(N_t+1)} _s)/ \sigma_{N_t}(\tau)]] \\
&=& \sum_{k=0}^{+\infty} E[\exp( I_{1}(t-\tau_{2k}))]p_{2k}(t)+ \sum_{k=0}^{+\infty} E[\exp(I_{2}(t-\tau_{2k+1}))]p_{2k+1}(t)
\end{eqnarray*}
Moreover:
\begin{eqnarray*}
E [ e^{I_{1}(v)}]&=& \lambda_{12} \int_0^{+ \infty} e^{I_1(x)-\lambda_{12} x}\;dx := \lambda_{12} J_1 \\
E [ e^{I_{2}(v*)}]&=& \lambda_{21} \int_0^{+ \infty} e^{I_2(x)-\lambda_{21} x}\;dx := \lambda_{21} J_2
\end{eqnarray*}
On the other hand:
\begin{eqnarray*}
E[e^{I_1(t-\tau_{2k})}] &=&  \int_0^{t} e^{I_1(x)}f_{t-\tau_{2k}}(x)\;dx  \\
&=& D(k, -k, 2k)  \int_0^{t} e^{I_1(x)}  (t-x)^{2k-1}e^{-\lambda_{21}  (t-x)}M(k,2k,(\lambda_{21}-\lambda_{12}) (t-x)) \;dx
\end{eqnarray*}
The previous equation, after the change of variable $z=t-x, dz=-dx$ reads:
\begin{eqnarray*}
&=&  D(k, -k, 2k)  \int_0^{t} e^{I_1(t-z)}  z^{2k-1}e^{-\lambda_{21} z}M(k,2k,(\lambda_{21}-\lambda_{12}) z) \;dz \\
&=&  D(k, -k, 2k) \sum_{l=0}^{+\infty} \frac{k^{(l)}(\lambda_{21}-\lambda_{12})^l}{(2k)^{(l)}l!}\int_0^{t} e^{I_1(t-z)}  z^{2k+l-1}e^{-\lambda_{21}  z} \;dz \\
&=&  D(k, -k, 2k) \sum_{l=0}^{+\infty}C(k,2k,l) G_1(t,2k+l)\lambda^l_{21} := J_3(t,k)
\end{eqnarray*}
Similarly:
\begin{eqnarray*}
E[e^{I_2(t-\tau_{2k+1})}] &=&  \int_0^{t} e^{I_2(x)}f_{t-\tau_{2k+1}}(x)\;dx  \\
&=& D(k+1, -k, 2k+1)  \int_0^{t} e^{I_2(x)}  (t-x)^{2k}e^{-\lambda_{21}  (t-x)}M(k+1,2k+1,(\lambda_{21}-\lambda_{12}) (t-x)) \;dx \\
&=&  D(k+1, -k, 2k+1) \sum_{l=0}^{+\infty} \frac{(k+1)^{(l)}(\lambda_{21}-\lambda_{12})^l}{(2k+1)^{(l)}l!}\int_0^{t} e^{I_2(t-z)}  z^{2k+l}e^{-\lambda_{21}  z} \;dz\\
&=&  D(k+1, -k, 2k+1)  \sum_{l=0}^{+\infty} C(k+1,2k+1,l)  G_2(t,2k+l+1)\lambda^l_{21} := J_4(t,k)
\end{eqnarray*}
Finally, from the independence of the process $(\xi_t)_{t \geq 0}$ on non-overlapping intervals:
\begin{eqnarray*}
&& E [\exp( i \int_0^t f(s)\;d \xi_s)] = E[\exp( i \sum_{j=1}^{N_t} \int_{\tau_{j-1}}^{\tau_{j}} f(s)\;  d\xi^{\nu(j)}_s)]E[\exp( i \int_{\tau_{N_t}}^{t}  f(s) d\xi^{\nu(N_t+1)}_s)] \\
&=& \left( \sum_{k=0}^{+\infty} (E[\exp(  I_{1}(v))])^k (E[\exp(  I_{2}(v^*))])^kp_{2k}(t) \right.\\
 &+& \left. \sum_{k=0}^{+\infty} (E[\exp(  I_{1}(v))])^{k+1} (E[\exp(  I_{2}(v^*))])^kp_{2k+1}(t) \right)\\
 && \left( \sum_{k=0}^{+\infty} E[\exp( I_{1}(t-\tau_{2k}))]p_{2k}(t)+ \sum_{k=0}^{+\infty} E[\exp(I_{2}(t-\tau_{2k+1}))]p_{2k+1}(t)\right)
\end{eqnarray*}
from which equation (\ref{eq:funclevy3}) immediately follows.
\end{proof}
A straightforward consequence of the previous theorem establishes  the characteristic function of the temperature process under the historic measure $P$.
\begin{corollary}\label{prop:chftempthetsdet}
For the model described by equations (\ref{eq:tempdyna}), (\ref{eq:seass2}) and (\ref{eq:tchlevy2})   the  characteristic function of $T_t$ under the probability $P$ is:
\begin{eqnarray}\label{eq:chftdet}
 \varphi_{T_t}(u) &=&  \exp[i u C_1(t, \alpha)]\varphi_{W_t}(u \sigma e^{-\alpha t})
\end{eqnarray}
 where the characteristic function of $W_t$, denoted by $\varphi_{W_t}(u)$, is given by equation (\ref{eq:funclevy3}) in Theorem \ref{teo:chfintpro} applied to   $\xi^j_t=V^j_t, f(s)=u \sigma e^{-\alpha(t-s)}$ and
\begin{equation*}
  I_j(u,x):= I_j(x)=\int_0^x l_{R^{j}}(-iu \mu^j_1 \sigma e^{-\alpha (t-s)}-\frac{1}{2}u^2\sigma^2 e^{-2 \alpha (t-s)} )\;  ds, \;\;j=1,2
\end{equation*}
\end{corollary}
\begin{proof}
First, notice that:
\begin{eqnarray*}\nonumber
  \varphi_{V^j_t}(u)&=& E[E[\exp(i (u V^j_t)/R^j_t)] ]=E[ \exp(i u \mu^j_1 R^j_t) E[\exp( iu B_{R^j_t}/R^j_t)] ]\\
  &=& E[\exp(i u \mu^j_1 R^j_t) \exp( -\frac{1}{2} R^j_t u^2)]= E[\exp(i( u \mu^j_1+\frac{1}{2}i  u^2)R^j_t)]\\
  &=&  \varphi_{R^j_t}(u \mu^j_1+\frac{1}{2}i  u^2)
\end{eqnarray*}
Hence:
\begin{equation}\label{eq:lv}
  l_{V^j}(u)=l_{R^j}(u \mu^j_1+\frac{1}{2}  u^2),\;j=1,2.
\end{equation}
By Lemma 2 and  equation (\ref{eq:funclevy}):
\begin{eqnarray*}\nonumber
  \varphi_{T_t}(u)&=& E[e^{iu T_t}]=\exp [iu C_1(t, \alpha)] E[\exp(iu  \sigma e^{-\alpha t} W_t)]= \exp[i u C_1(t, \alpha)]\varphi_{W_t}(u \sigma e^{-\alpha t})
\end{eqnarray*}
\end{proof}
\section{Switching model under an Esscher transformation}
In order to select the EMM for pricing purposes we consider an Esscher transform of the historic measure $P$. See \cite{EsscherTransform} for a rationale of its use in terms of a utility-maximization wealth criterion. \\
Thus, for a stochastic process $(X_t)_{t \geq 0}$ we consider a change of probability given by the Radon-Nicodym derivative:
  \begin{equation}\label{eq:esscher}
  \frac{d \mathcal{Q}^{\theta}_t}{d P_t}=\\exp[\theta X_t- l_{X_t}(\theta)],\; \theta \in \mathbb{R}
\end{equation}
   where $P_t$ and $\mathcal{Q}^{\theta}_t$ are the respective restrictions of $P$ and $\mathcal{Q}^{\theta}$ to the $\sigma$-algebra $\mathcal{F}_t$. The functions $\varphi^{\theta}_{X_t}(u)$ and $ l^{\theta}_X(u)$  define  respectively  the characteristic function and  the cumulant generating function  of a process $(X_t)_{t \geq 0}$ under the probability $\mathcal{Q}^{\theta}$.\\
      For consistency we denote $\varphi^{0}_{X_{t}}:=\varphi_{X_t}$ and $ l_X^{0}=l_X$.\\
By analogy with the case of financial underlying assets the risk-neutral measure $\mathcal{Q}^{\theta}$ is set to make the discounted temperatures process $(\tilde{T}_t)_{t \geq 0}$ a martingale. The expected value under $\mathcal{Q}^{\theta}$ is denoted $E_{\theta}$.\\
 Under an Esscher transform of parameter $\theta$ the probabilities of the number of regime changes in the interval $[0,t)$ are:
\begin{eqnarray*}\label{eq:probtheta1}
p^{\theta}_k(t) &=& P^{\theta}(N_t=k) = \frac{e^{\theta k } P(N_t=k)}{M_{N_t}(\theta)},\; k \in \mathbb{N}
\end{eqnarray*}
where $M_{N_t}(\theta)=E[e^{\theta N_t}]$ its moment generating function(m.g.f.)  of $N_t$ under the probability \textit{P}.\\
The next result provides an expression for the characteristic function of the temperature process under an Esscher transformation of parameter $\theta$. To this end we let the notations introduced in Theorem \ref{teo:chfintpro} depend on the Esscher parameter $\theta$. Hence:
\begin{eqnarray*}
 I^{\theta}_j(u,x)&=& \int_0^x l_{R^{j}}(g_j(u,s, \theta))\;  ds-l_{R^j}(\mu^j_1 \theta+\frac{1}{2} \theta^2))x \\
G^{\theta}_j(u,t,m) &=& \int_0^{t} e^{I^{\theta}_j(u,t-z)}  z^{m-1}e^{-(\lambda_{21}-\theta)z} \;dz,\; j=1,2 \\
 J^{\theta}_1(u)&=& \int_0^{+\infty} e^{I^{\theta}_{1}(u,x)}e^{-(\lambda_{12}-\theta)x}\; dx,\;  J^{\theta}_2(u)= \int_0^{+\infty} e^{I^{\theta}_{2}(u,x)}e^{-(\lambda_{21}-\theta)x}\; dx \\
 J^{\theta}_3(u,k)&=& D(k, -k, 2k, \theta) \sum_{l=0}^{+\infty}C(k,2k,l) G^{\theta}_1(u,t,2k+l)\lambda^l_{21} \\
 J^{\theta}_4(u,k) &=& D(k+1, -k, 2k+1,\theta)\sum_{l=0}^{+\infty}C(k+1,2k+1,l) G^{\theta}_2(u,t,2k+l+1) \lambda^l_{21}
\end{eqnarray*}
 \begin{theorem}\label{prop:chfuntheta}
Under the probability $\mathcal{Q}^{\theta}$ defined by the Esscher transform  in equation (\ref{eq:esscher}) and the condition $\lambda_{21} > \lambda_{12} > \theta$ the characteristic function of the temperature process is given by:
 \begin{eqnarray}\label{eq:chfuntemptheta}
        \varphi^{\theta}_{T_t}(u)&=& \exp[iu C_1(t, \alpha)]  C_4(u, \theta)C_5(u,\theta)
 \end{eqnarray}
where:
\begin{eqnarray*}
 D(m,n,l, \theta)&=& \frac{(\lambda^{ m}_{12}-\theta)}{(\lambda^{ n}_{21}-\theta) \Gamma(l)}, \; l \geq 1 \\
C_2(u, \theta, k) &=& (\lambda_{12}-\theta)^k (\lambda_{21}-\theta)^k (J^{\theta}_1(u)J^{\theta}_2(u))^k \\
C_3(u, \theta, k) &=& p^{\theta}_{2k}(t)+ (\lambda_{12}-\theta)J^{\theta}_1(u) p^{\theta}_{2k+1}(t)\\
C_4(u, \theta) &=& \sum_{k=0}^{+\infty}  C_2(u, \theta,k) C_3(u, \theta, k)\\
C_5(u,\theta) &=& \sum_{k=0}^{+\infty}   J^{\theta}_3(u,k)  p^{\theta}_{2k}(t) +  \sum_{k=0}^{+\infty}   J^{\theta}_4(u,k)p^{\theta}_{2k+1}(t) \\
\end{eqnarray*}
\begin{equation*}
  g_j(u,s, \theta)=-iu \mu^j_1 \sigma e^{-\alpha (t-s)} +  \mu^j_1 \theta +\frac{1}{2}(-i u\sigma e^{- \alpha (t-s)}+\theta)^2
\end{equation*}
\end{theorem}
 \begin{proof}
 We follow the lines of  Theorem \ref{teo:chfintpro} applied to $\xi^j_t=V^j_t$ and $f(s)=u \sigma e^{-\alpha(t-s)}$.\\
First, notice that from the change of probability defined by the Esscher transform:
   \begin{eqnarray*}
    \varphi^{\theta}_{V^j_t}(u) &=&  E(e^{i u V^j_t} e^{\theta V^j_t- l_{V^j_t}(\theta)})=\frac{\varphi_{V^j_t}(u- i \theta)}{\varphi_{V^j_t}(- i \theta)}\\
   \end{eqnarray*}
   and $ l^{\theta}_{V^j}(u) = l_{V^j}(u+\theta)-l_{V^j}(\theta)$.\\
  The m.g.f. of the number of regime changes on $[0,t)$ under the Esscher transform becomes:
\begin{eqnarray*}
M_{N_t}(\theta) &=& \sum_{k=0}^{+\infty}  e^{k \theta } p_k(t)=\sum_{k=0}^{+\infty}  e^{2 k \theta  }p_{2k}(t)+\sum_{k=0}^{+\infty}  e^{(2k+1) \theta }p_{2k+1}(t)\\
&=& \sum_{k=0}^{+\infty}  e^{2 k \theta   } \left(p_{2k}(t) +  e^{\theta}p_{2k+1}(t) \right) \\
&=& \sum_{k=0}^{+\infty}  e^{2 k \theta   } \left( F_{\tau_{2k}}(t)+  (e^{\theta}-1) F_{\tau_{2k+1}}(t)- e^{\theta}F_{\tau_{2k+2}}(t) \right)
\end{eqnarray*}
Furthermore, under the measure $\mathcal{Q}^{\theta}$ the times between two consecutive regime changes are  independent random variables denoted by $v_j$ with exponential distribution of parameters $\lambda_{12}-\theta$ and $\lambda_{21}-\theta$ if the process is in regime one or two respectively.\\
The p.d.f's  of the random variables $\tau_{2k}$ and $\tau_{2k+1}$ under $\mathcal{Q}^{\theta}$, denoted respectively by $f_{\tau_{2k}}(x, \theta)$ and $f_{\tau_{2k+1}}(x, \theta)$ are given  by:
\begin{eqnarray*}
f_{\tau_{2k}}(x, \theta) &=& D(k,-k,2k, \theta)  M(k,2k,(\lambda_{21}-\lambda_{12})x)x^{2k-1}e^{-(\lambda_{21}-{\theta}) x} \\
f_{\tau_{2k+1}}(x, \theta) &=& D(k+1,-k,2k+1,\theta)  M(k+1,2k+1,(\lambda_{21}-\lambda_{12})x)  x^{2k}e^{-(\lambda_{21}-{\theta}) x}
\end{eqnarray*}
for $ x \geq 0,  \; k \in \mathbb{N}$. See Lemma \ref{lemma:tausdistrib} and the definition of an Esscher transform \\
 Hence:
   \begin{eqnarray*}
&& E_{\theta} [\exp( i \int_0^t f(s)\;d \xi_s)] = E_{\theta}[\exp( i \sum_{j=1}^{N_t} \int_{\tau_{j-1}}^{\tau_{j}} f(s)\;  d\xi^{\nu(j)}_s)]E_{\theta}[\exp( i \int_{\tau_{N_t}}^{t}  f(s) d\xi^{\nu(N_t+1)} _s)] \\
&=& \left( \sum_{k=0}^{+\infty} (E_{\theta}[\exp(  I^{\theta}_{1}(u,v))])^k (E_{\theta}[\exp(  I^{\theta}_{2}(u,v^*))])^k p^{\theta}_{2k}(t) \right.\\
 &+& \left. \sum_{k=0}^{+\infty} (E_{\theta}[\exp(  I^{\theta}_{1}(u,v))])^{k+1} (E_{\theta}[\exp(  I^{\theta}_{2}(u,v^*))])^k p^{\theta}_{2k+1}(t) \right)\\
 && \left( \sum_{k=0}^{+\infty} E_{\theta}[\exp( I^{\theta}_{1}(u,t-\tau_{2k}))]p^{\theta}_{2k}(t)+ \sum_{k=0}^{+\infty} E_{\theta}[\exp(I^{\theta}_{2}(u,t-\tau_{2k+1}))]p^{\theta}_{2k+1}(t)\right)
\end{eqnarray*}
 with
 \begin{eqnarray*}
  E_{\theta}[\exp(  I^{\theta}_{1}(u,v))] &=&  (\lambda_{12}-\theta) J^{\theta}_1(u) \\
  E_{\theta}[\exp(  I^{\theta}_{2}(u,v^*))  &=& (\lambda_{21}-\theta)J^{\theta}_2(u)
 \end{eqnarray*}
 and
 \begin{eqnarray*}
 &&  E_{\theta}[\exp( I^{\theta}_{1}(u,t-\tau_{2k}))] =  \int_0^{t} e^{I^{\theta}_1(u,x)}f_{t-\tau_{2k}}(x,{\theta})\;dx = \int_0^{t} e^{I^{\theta}_1(u,x)}f_{\tau_{2k}}(t-x,{\theta})\;dx \\
&=& D(k,-k,2k, \theta) \int_0^{t} e^{I^{\theta}_1(u,x)}  (t-x)^{2k-1}e^{-(\lambda_{21}-{\theta}) (t-x)} M(k,2k,(\lambda_{21}-\lambda_{12})(t-x)) \;dx \\
&=& D(k,-k,2k, \theta) \sum_{l=0}^{+\infty} C(k,2k,l)G^{\theta}_1(u,t,2k+l)\lambda^l_{21}= J^{\theta}_3(u,k)
 \end{eqnarray*}
 Similarly:
 \begin{eqnarray*}
 E_{\theta}[\exp(I^{\theta}_{2}(u,t-\tau_{2k+1}))] &=& D(k+1,-k,2k+1, \theta) \sum_{l=0}^{+\infty} C(k+1,2k+1,l)G^{\theta}_2(u,t,2k+l+1)\lambda^l_{21}\\
 &=& J^{\theta}_4(u,k)
 \end{eqnarray*}
  Moreover:
\begin{eqnarray*}
  && \int_0^x l^{\theta}_{V^{j}}(-iu \sigma e^{-\alpha (t-s)})\;  ds=\int_0^x l_{V^{j}}(-iu \sigma e^{-\alpha (t-s)}+\theta)\;  ds - l_{V^j}(\theta)x \\
  &=&   \int_0^x l_{R^{j}}(g_j(u,s, \theta))\;  ds - l_{R^j}(\mu^j_1 \theta + \frac{1}{2} \theta^2)x= I^{\theta}_j(u,x)
\end{eqnarray*}
which leads to equation (\ref{eq:chfuntemptheta}).
 \end{proof}
Motivated by the search of a risk-neutral environment in the pricing of weather derivatives  the next result specifies the value of $\theta$ under the Esscher transform to obtain an EMM. For practical proposes we restrain the EMM condition to the case $s=0$ and any $t \geq 0$. Hence, the martingale condition reads $E[\tilde{T}_t]= \tilde{T}_0$. The general case follows in a similar way.\\
We  introduce the following quantities:
\begin{eqnarray*}
 L_j(\theta, \alpha)&=& \alpha^{-1}(\mu^j_1+\theta) \frac{d l_{R^{j}}}{du}(\mu^j_1 \theta+ \frac{1}{2}\theta^2)\\
  A_{12}(\theta, \alpha) &=& \frac{  L_1(\theta, \alpha)}{(\lambda_{12}-\alpha -\theta)}\\
   A_{21}(\theta, \alpha) &=&  \frac{L_2(\theta, \alpha, \sigma)}{(\lambda_{21}-\alpha -\theta)}
\end{eqnarray*}
 and the functions $I^{\theta}_j(u,x), J^{\theta}_l(u), G^{\theta}_j(u,t,m), J_3(u,t,k, \theta)$ and $J_4(u,t,k, \theta)$ are defined as in Theorem \ref{prop:chfuntheta}.
   \begin{theorem}\label{prop:chfuntheta2}
 Let $(T_t)_{t \geq 0}$ be the temperature process given by equations (\ref{eq:tempdyna}), (\ref{eq:seass2}) and (\ref{eq:tchlevy2}) starting at regime one. Then,  the Esscher measure $\mathcal{Q}^{\theta}$ is an EMM  if  for any $t>0$ the parameter $\theta$ verifies under the conditions $\lambda_{12}>max(\alpha+\theta,0)$ and  $\lambda_{12} < \lambda_{21}$ the equation:
 \begin{equation}\label{eq:gershui}
R(t,\theta)= \sigma^{-1} e^{\alpha t}(e^{rt} T_0-C_1(t, \alpha))
   \end{equation}
   where:
  \begin{eqnarray*}
 R(t,\theta)&=& -i  \left[ D_1(t,\theta) C_5(0, \theta)+ D_2(t,\theta) \right]\\
  \end{eqnarray*}
  The expressions $D_1(t,\theta)$ and $D_2(t,\theta)$  are given  by:
  \begin{eqnarray*}
  D_1(t,\theta) &=&  -i \alpha \left[\left(A_{12}(\theta, \alpha)+A_{21}(\theta, \alpha)\right)E_{\theta}[N_t]+
 A_{12}(\theta, \alpha) p^{\theta}_O(t) \right] \\ \label{eq:D1F}
 && \\
 D_2(t,\theta) &:=&  \frac{dC_5(u \sigma^{1}e^{\alpha t}, \theta,k)}{du}|_{u=0} = \sum_{k=0}^{+\infty}   \frac{dJ^{\theta}_3(u \sigma^{-1}e^{\alpha t},k)}{du}|_{u=0}  p^{\theta}_{2k}(t)\\ \nonumber
  &+&  \sum_{k=0}^{+\infty}  \frac{dJ^{\theta}_4(u \sigma^{-1}e^{\alpha t},k)}{du}|_{u=0}p^{\theta}_{2k+1}(t) \\ \label{eq:D2F}
 &&
  \end{eqnarray*}
 where $p^{\theta}_{O}(t)$ is the probability to observe an odd number of regime changes in the  interval $[0,t)$.
 \end{theorem}
\begin{proof}
From Lemma \ref{le:ito} the discounted temperature process $(\tilde{T}_t)_{t \geq 0}$ verifies:
     \begin{equation*}
       \tilde{T}_t= \tilde{C}_1(t, \alpha)+\sigma e^{-\alpha t} \tilde{W}_t
     \end{equation*}
   Then, the EMM condition $E_{\theta}(\tilde{T}_t)=T_0$ translates into the equivalent condition:
   \begin{eqnarray*}
    C_1(t, \alpha)+ \sigma e^{-\alpha t}E_{\theta}[W_t]=e^{rt}T_0
   \end{eqnarray*}
   or equivalently the parameter $\theta$ solves:
   \begin{eqnarray}\label{eq:frtmomw}
   E_{\theta}[W_t] &=&  \sigma^{-1} e^{\alpha t}(e^{rt}T_0-C_1(t, \alpha))
   \end{eqnarray}
  We invoke Theorem \ref{prop:chfuntheta} to compute the first moment of $W_t$ via its characteristic function. To this end notice that
  \begin{equation*}
    W_t=\sigma^{-1} e^{\alpha t}(T_t-C_1(t, \alpha))
  \end{equation*}
  Consequently:
    \begin{eqnarray*}
  \varphi^{\theta}_{W_t}(u) &=& \exp(-iu \sigma^{-1}e^{\alpha t}C_1(t, \alpha)) \varphi^{\theta}_{T_t}(u \sigma^{-1}e^{\alpha t})\\
    &=& \exp[-iu \sigma^{-1}e^{\alpha t}C_1(t, \alpha)] \exp[iu \sigma^{-1}e^{\alpha t}C_1(t, \alpha)] C_4(u \sigma^{-1}e^{\alpha t}, \theta)  C_5(u \sigma^{-1}e^{\alpha t}, \theta)\\
    &=& C_4(u \sigma^{-1}e^{\alpha t}, \theta)  C_5(u \sigma^{-1}e^{\alpha t}, \theta)
  \end{eqnarray*}
 Notice that $C_2(0, \theta,k)=1$ and $C_3(0, \theta,k)=p^{\theta}_{2k}(t)+p^{\theta}_{2k+1}(t)$.\\
 Computation of the derivatives   of the intermediate functions $I^{\theta}_j(u,x), J^{\theta}_j(u), G^{\theta}_j(u,t,m)$, $J^{\theta}_3(u,t,k),J^{\theta}_4(u,t,k), C_j(u, \theta)$ is straightforward. It is left to the appendix. \\
  Notice also that:
 \begin{eqnarray*}
 G^{\theta}_j(0,t,m) &=& \int_0^{t}   z^{m-1}e^{-(\lambda_{21}-\theta)z} \;dz= (\lambda_{21}-\theta)^{-m} \gamma(m,(\lambda_{21}-\theta)t)\\
  J^{\theta}_3(0,k)&=& D(k, -k, 2k,\theta) \sum_{l=0}^{+\infty}C(k,2k,l)G^{\theta}_j(0,t,2k+l) \lambda^l_{21} \\
 J^{\theta}_4(0,k) &=& D(k+1, -k, 2k+1,\theta)\sum_{l=0}^{+\infty}C(k+1,2k+1,l)G^{\theta}_j(0,t,2k+l+1) \lambda^l_{21} \\
C_2(0, \theta, k) &=& 1 \\
C_3(0, \theta, k) &=& p^{\theta}_{2k}(t)+  p^{\theta}_{2k+1}(t)
\end{eqnarray*}
Leading to:
 \begin{eqnarray} \nonumber
 C_4(0, \theta) &=&  \sum_{k=0}^{+\infty}  C_2(0, \theta,k) C_3(0, \theta, k)= \sum_{k=0}^{+\infty}  C_3(0, \theta, k)\\ \nonumber
 &=& \sum_{k=0}^{+\infty}  (p^{\theta}_{2k}(t)+p^{\theta}_{2k+1}(t))=1 \\ \nonumber
 C_5(0, \theta) &=& \sum_{k=0}^{+\infty}   J^{\theta}_3(0,k)  p^{\theta}_{2k}(t) +  \sum_{k=0}^{+\infty}   J^{\theta}_4(0,k)p^{\theta}_{2k+1}(t)\\ \label{eq:c5zero}
 &&
 \end{eqnarray}
 Hence,
  \begin{eqnarray*}
  \frac{d \varphi_{W^{\theta}_t}(u)}{du}|_{u=0} &=& \frac{d C_4(u \sigma^{-1} e^{\alpha t}, \theta)}{du}|_{u=0} C_5(0, \theta)+ \frac{d C_5(u \sigma^{-1} e^{\alpha t}, \theta)}{du}|_{u=0} C_4(0, \theta)\\
    &=& D_1(t,\theta) C_5(0, \theta)+ D_2(t,\theta)
  \end{eqnarray*}
  This last result combined with equation (\ref{eq:frtmomw}) and taking account that $E_{\theta}[W_t]=-i \frac{d \varphi_{W^{\theta}_t}(u)}{du}|_{u=0}$,
   leads to equation (\ref{eq:gershui}).
 \end{proof}
 \begin{figure}[htb!]
\centering
\includegraphics[width=\textwidth]{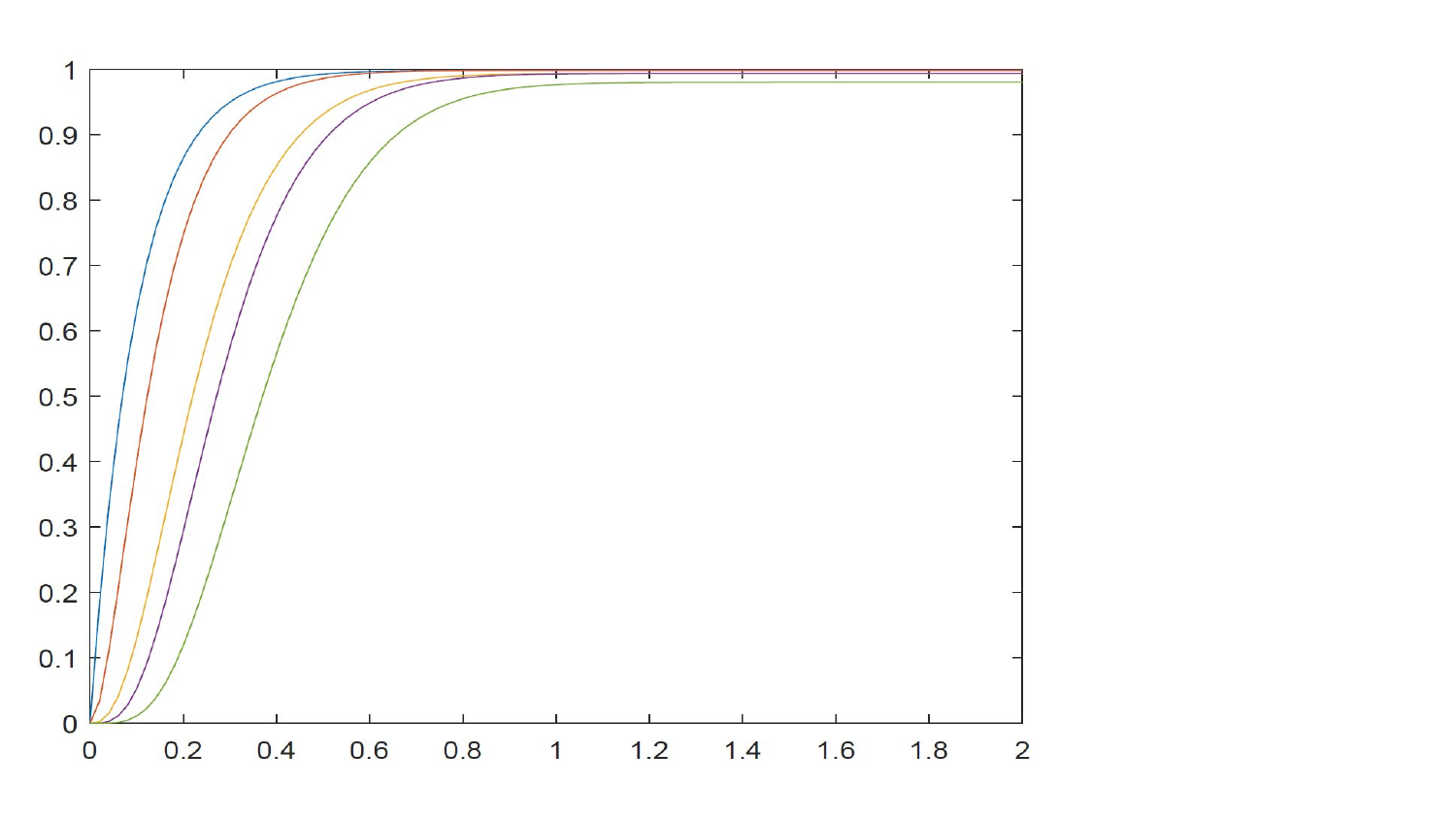}
\caption{C.d.f's of the times of regime changes. The blue, red, yellow, magenta and green  curves represent respectively the cases $k=1,2,\ldots,5$ }
\label{fig:cdfregchanges}
\end{figure}

\begin{figure}[htb!]
\centering
\includegraphics[width=\textwidth]{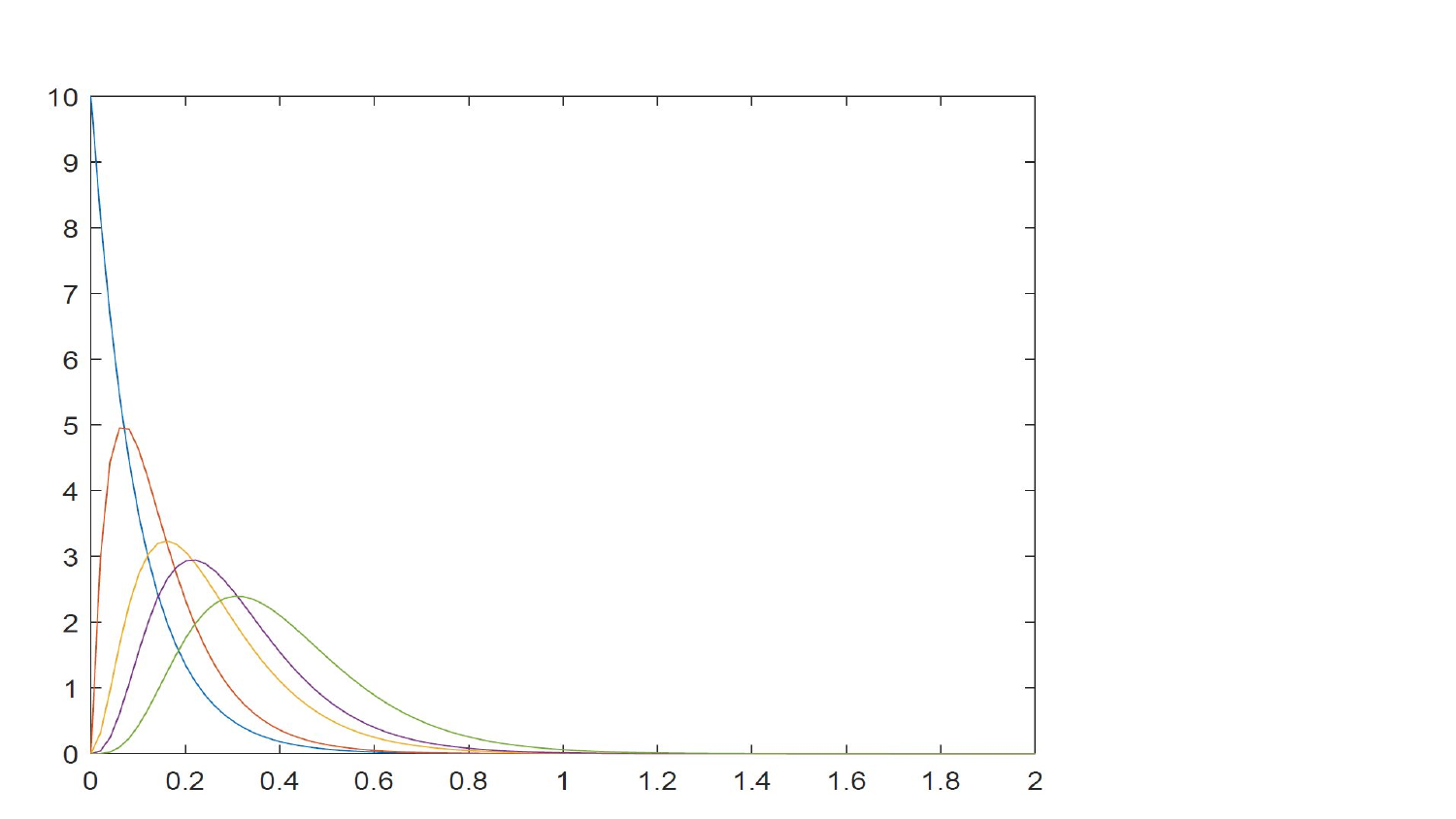}
\caption{P.d.f's of the times of regime changes. The blue, red, yellow, magenta and green  curves represent respectively the cases $k=1,2,\ldots,5$ }
\label{fig:pdfregchg}
\end{figure}

\section{Examples and numerical results}
First, the p.d.f's of the times of regime changes, denoted by $\tau_k$, are shown in Figure \ref{fig:pdfregchg} for $k=1,2,\ldots,5$. The blue curve represents the case $k=1$, which is the p.d.f. of an exponential distribution with parameter $\lambda_{12}$, and the green curve represents the case $k=5$, with cases $k=2,3,4$ in between. The parameter values, chosen for illustrative proposes, are  $\lambda_{12}=10$ and $\lambda_{21}=20$. The rationale in the selection is to consider a stable regime describe by $\lambda_{12}$ and an unstable one given by $\lambda_{21}$. Thus, the system expends on average twice of the time in regime one.\\
In Figure \ref{fig:cdfregchanges}  the c.d.f's of the times of regime changes according to the expressions found in Lemma \ref{lemma:tausdistrib} are represented. The truncation  value in the series is chosen as $n=10$, i.e. including ten terms in the summation, results in a reasonable approximation for the  rate parameters $\lambda_{12}$ and $\lambda_{21}$.\\
In Figure \ref{fig:probregchg} are shown the probabilities to observe $k$ regime changes over an interval $[0,t)$ as described by expression (\ref{eq:poisson1}). Calculations are taken over time horizons $t=1/12,1/4,1$ corresponding to a month, a quarter and a year respectively. Notice that the values of the probability changes depending whether the number of regime changes is even or odd. It is  consistent with the choice of the rate parameters leading to switching between a stable and an unstable period.\\

\begin{figure}[htb!]
\begin{center}
\subfigure[]{
\resizebox*{7cm}{!}{\includegraphics{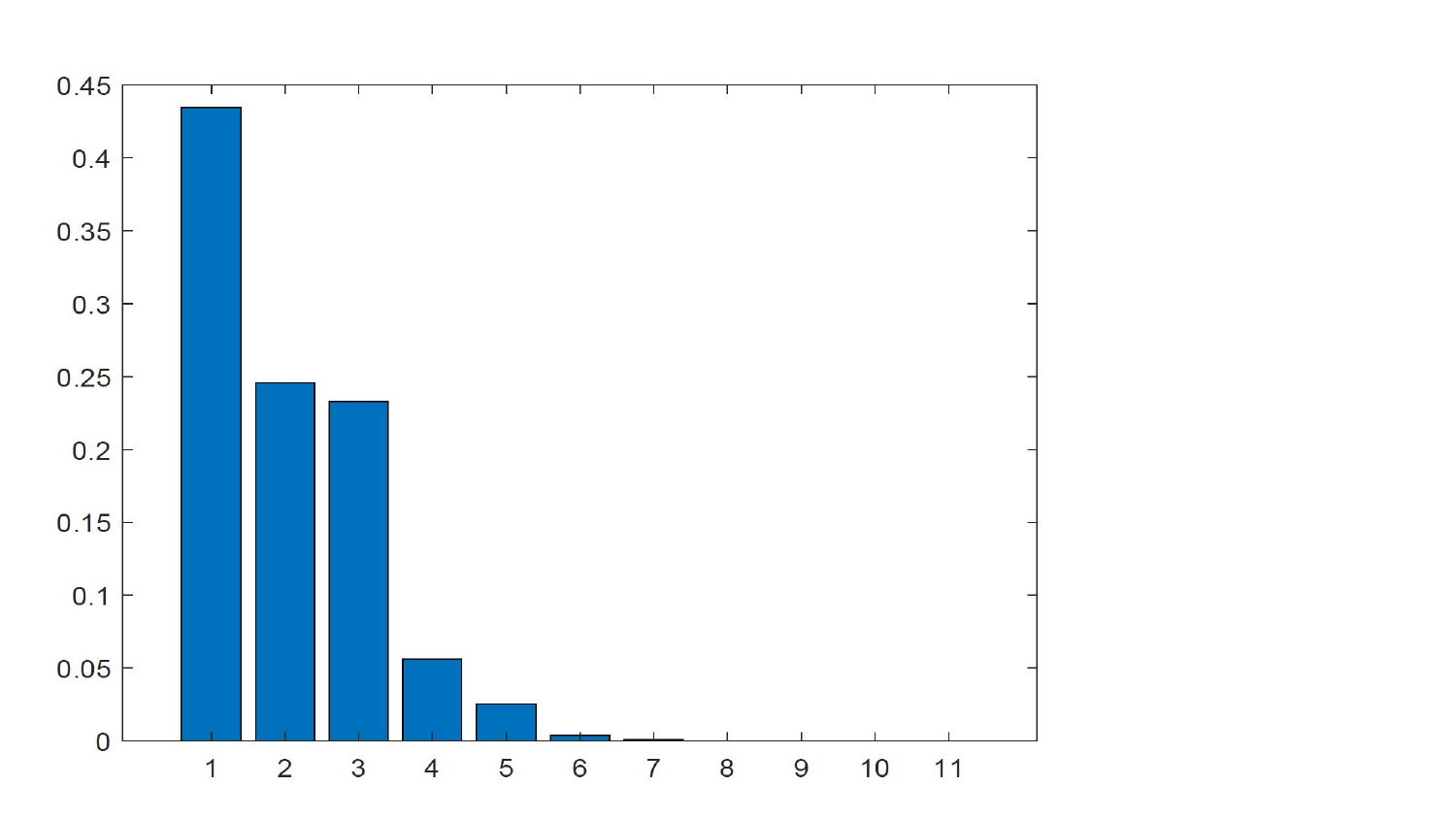}}}
\subfigure[]{
\resizebox*{7cm}{!}{\includegraphics{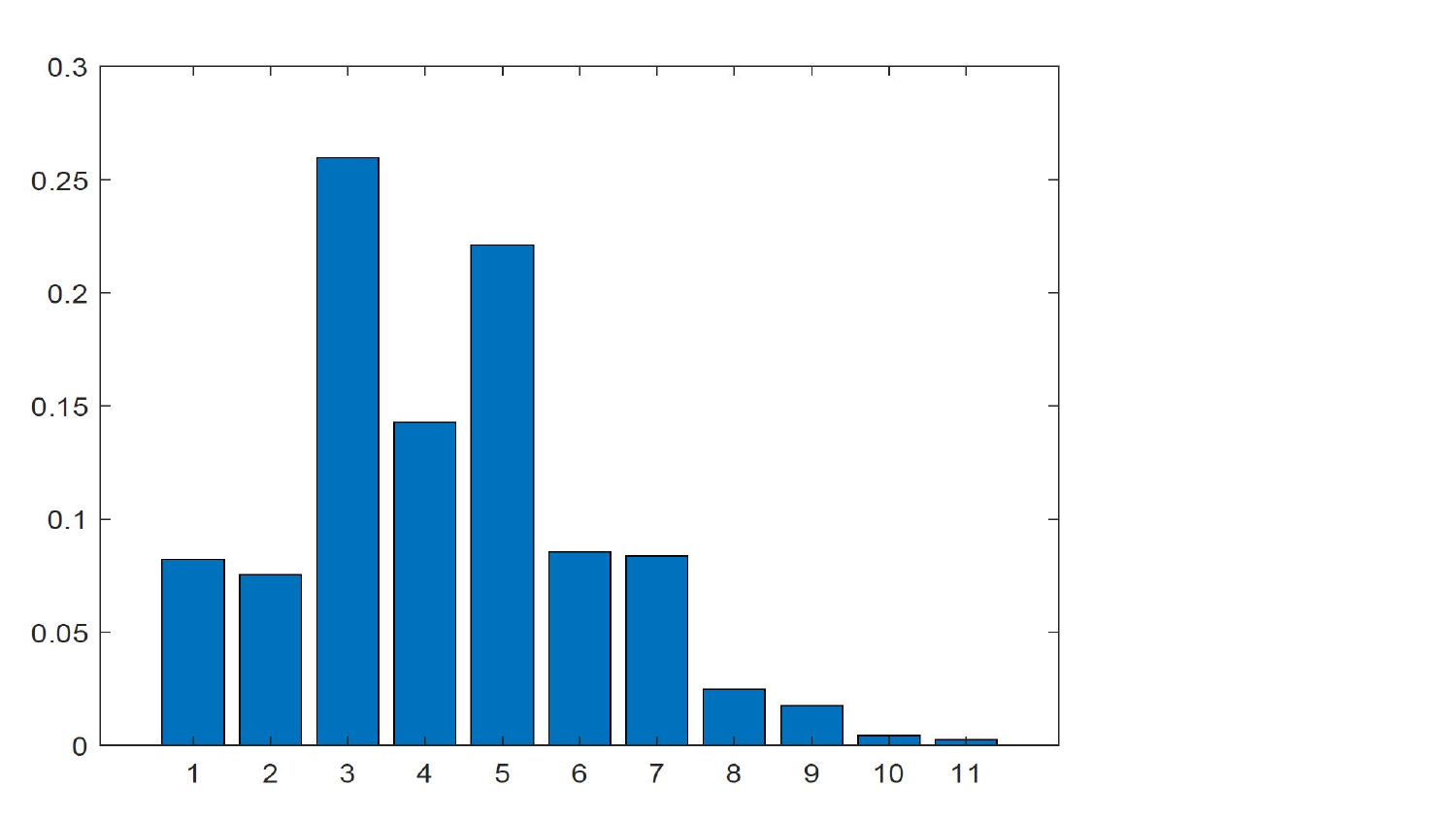}}}\hspace{5pt}
\subfigure[]{ \resizebox*{7cm}{!}{\includegraphics{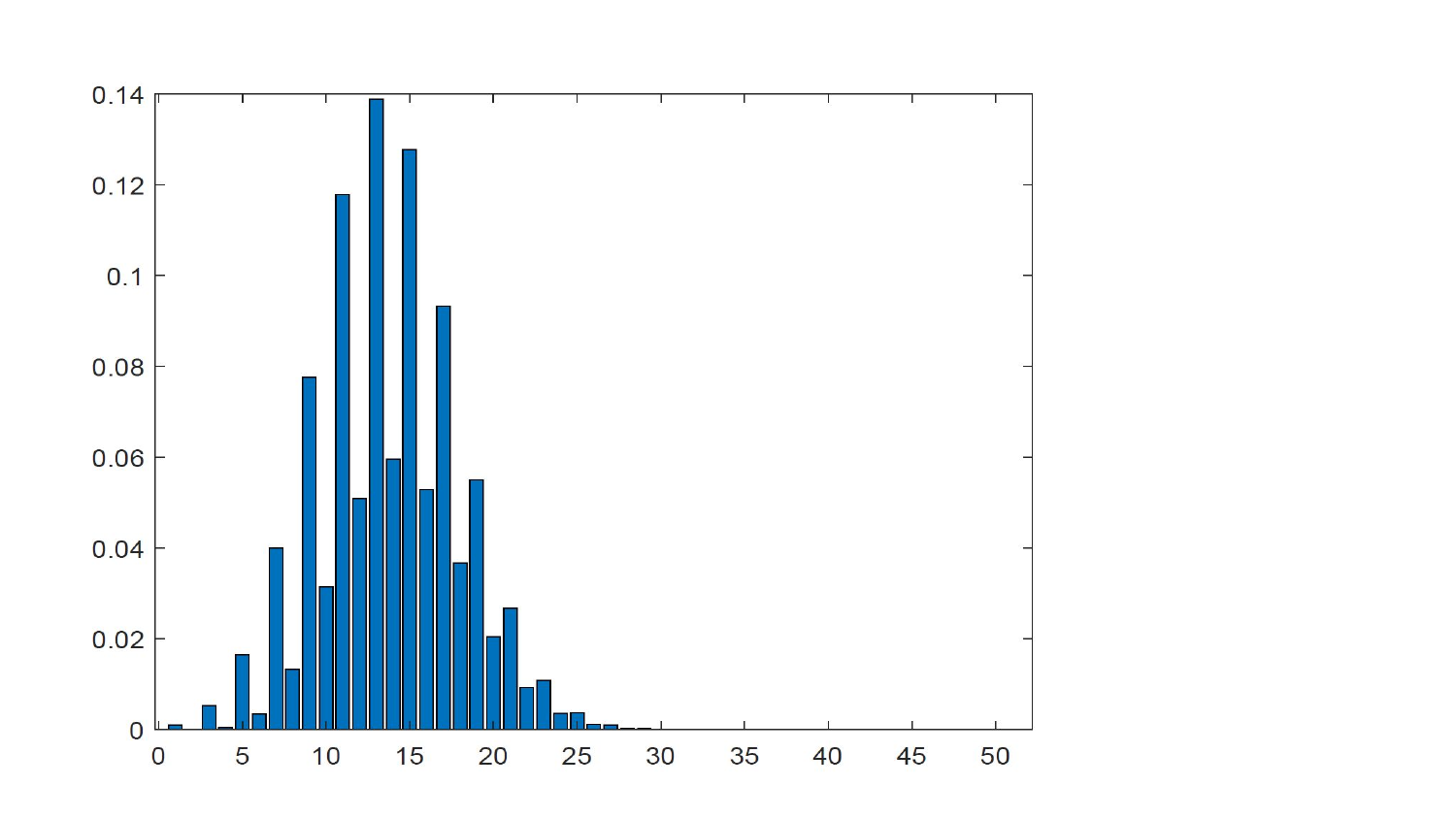}}}
 \caption{Probabilities to observe $k$ regime changes over an interval $[0,t)$ with  $t=1/12,1/4,1$. }\label{fig:probregchg}
\end{center}
\end{figure}

Next, we analyze in details the calculation of the characteristic function and the Esscher parameter in the case of the model (\ref{eq:tempdyna}) with Gamma and Inverse Gaussian subordinators. \\
As there is not closed-form formula of the characteristic function, numerical approximations are required. They involve  truncations in the probability distribution of the changes of regimes on $[0,t]$, as well as in other intermediate related series. Numerical integration, following trapezoidal rule, is also necessary in repeated occasions. Generally speaking, the order of calculation  to obtain the characteristic function under the Esscher transform is the following:
\begin{eqnarray*}
D,C,\Gamma,B_j, M_{\theta}(k), P_{\theta}(k), l^{\theta}_{R^j},I^{\theta}_j, J^{\theta}_l,G^{\theta}_l, J^{\theta}_l(u,k),C_2, C_3, \varphi^{\theta}_{T_t}
\end{eqnarray*}

\begin{example}\textit{Gamma  subordinator }
Consider the subordinators $(R^j_t)_{t \geq 0}$  are Gamma processes with parameters $a_j >0,b_j >0, j=1,2$ with respective
characteristic function and Laplace exponent:
 \begin{eqnarray*}
 \varphi_{R^j}(u) &=& \left( 1- \frac{i u}{b_j} \right)^{-a_j t},\;a_j>0, b_j>0 \\
 l_{R^j}(u) &=&  -a_j \log \left( 1- \frac{ u}{b_j}\right),\; u<b_j
 \end{eqnarray*}
 and
 \begin{eqnarray*}
 \varphi^{\theta}_{R^j}(u) &=& \left( 1- \frac{i g_j(u,s,\theta)}{b_j} \right)^{-a_j t},\;a_j>0, b_j>0 \\
 l^{\theta}_{R^j}(u) &=&  -a_j \log \left( 1- \frac{ g_j(u,s,\theta) }{b_j}\right),\; u<b_j
 \end{eqnarray*}
 where:
 \begin{equation*}
  g_j(u,s,\theta)=u \mu^j_1 \sigma e^{-\alpha t}e^{\alpha s} +  \mu^j_1 \theta + \frac{1}{2}(u \sigma e^{-\alpha t}e^{\alpha s}+\theta)^2,\;j=1,2.
\end{equation*}
Therefore:
 \begin{eqnarray*}
 I^{\theta}_j(u,x)&=&  \int_0^x l_{R^{j}}g_j(u,s,\theta)\;  ds-l_{R^j}(\mu^j_1 \theta+\frac{1}{2} \theta^2)x \\
&=&  -a_j \left[ \int_0^x \log \left( 1- \frac{ g_j(u,s,\theta) }{b_j}\right) \;  ds-log \left(1-\frac{(\mu^j_1 \theta+\frac{1}{2} \theta^2)}{b_j} \right)x \right]\\
J^{\theta}_l(u)&=& \int_0^{+\infty} e^{I^{\theta}_{l}(u,x)}e^{-(\lambda_{l}-\theta)x}\; dx \\
&=& \left( 1-\frac{\mu^j_1 \theta+\frac{1}{2} \theta^2}{b_j} \right)^{-a_j x} \int_0^{+\infty} \exp(-a_j \int_0^x \log \left( 1- \frac{ g_j(u,s,\theta) }{b_j}\right) \;  ds-(\lambda_{l}-\theta)x )\; dx \\
G^{\theta}_j(u,m) &=& \int_0^{t} e^{I^{\theta}_j(u,t-z)}  z^{m-1}e^{-(\lambda_{2}-\theta)z} \;dz,\; j=1,2 \\
&=& \left(1-\frac{ \mu^j_1 \theta+\frac{1}{2} \theta^2}{b_j} \right)^{-a_j t}\\
&& \int_0^{t}\left(1-\frac{ \mu^j_1 \theta+\frac{1}{2} \theta^2}{b_j} \right)^{a_j z}\\
&& \exp( -a_j  \int_0^{t-z} \log \left( 1- \frac{ g_j(u,s,\theta) }{b_j}\right) \;  ds)z^{m-1}e^{-(\lambda_{2}-\theta)z} \;dz
 \end{eqnarray*}
  We compute the  Gerber-Shiu  parameter $\theta$ from the martingale condition given by equation (\ref{eq:gershui}). First,
  \begin{eqnarray*}
  \frac{d l_{R^{j}}(\mu^j_1 \theta+ \frac{1}{2}\theta^2)}{du}&=& \frac{a_j}{b_j-\mu^j_1 \theta-\frac{1}{2} \theta^2}\\
   L_j(\theta, \alpha)&=& \frac{(\mu^j+\theta)}{\alpha} \frac{d l_{R^{j}}(\mu^j_1 \theta+ \frac{1}{2}\theta^2)}{du}\\
&=& \frac{a_j (\mu^j+\theta)}{b_j-\mu^j_1 \theta-\frac{1}{2} \theta^2}
\end{eqnarray*}
   The previous calculations allow to find $D_1, D_2, C_4$ and $C_5$ and their derivatives leading to $R(t,\theta)$ in equation (\ref{eq:gershui}).
\end{example}
  \begin{example} \textit{Switching Inverse Gaussian  subordinator }\\
  \begin{eqnarray*}
 \varphi_{R^j}(u) &=& \exp(-a_j \sqrt{-2iu+b^2_j}-b_j) \\
 l_{R^j}(u) &=& -a_j \sqrt{-2u+b^2_j}-b_j
 \end{eqnarray*}
 and
 \begin{eqnarray*}
 \varphi^{\theta}_{R^j}(u) &=& \exp(-a_j \sqrt{-2i  g_j(u,s,\theta)+b^2_j}-b_j) \\
 l^{\theta}_{R^j}(u) &=& -a_j \sqrt{-2  g_j(u,s,\theta)+b^2_j}-b_j
 \end{eqnarray*}
\begin{eqnarray*}
 I^{\theta}_j(u,x)&=&  \int_0^x l_{R^{j}}(g_j(u,s,\theta))\;  ds-l_{R^j}(\mu^j_1 \theta+\frac{1}{2} \theta^2)x \\
&=&  -a_j \left[ \int_0^x  \sqrt{-2  g_j(u,s,\theta)+b^2_j}\;  ds -(b_j+\sqrt{-2  (\mu^j_1 \theta+\frac{1}{2} \theta^2)+b^2_j})x \right]\\
J^{\theta}_l(u)&=& \int_0^{+\infty} \exp(-a_j (\int_0^x  \sqrt{-2  g_j(u,s,\theta)+b^2_j}\;  ds)\exp(-(b_j+\sqrt{-2  (\mu^j_1 \theta+\frac{1}{2} \theta^2)+b^2_j}\lambda_{l}-\theta)x)\; dx \\
G^{\theta}_j(u,m) &=& \int_0^{t} e^{I^{\theta}_j(u,t-z)}  z^{m-1}e^{-(\lambda_{2}-\theta)z} \;dz \\
&=& \int_0^{t} \exp(-a_j  \int_0^{t-z}  \sqrt{-2  g_j(u,s,\theta)+b^2_j}\;  ds \\
&& -a_j(b_j+\sqrt{-2  (\mu^j_1 \theta+\frac{1}{2} \theta^2)+b^2_j})(t-z))  z^{m-1}e^{-(\lambda_{2}-\theta)z} \;dz \\
&=& \left(b_j+\sqrt{-2  (\mu^j_1 \theta+\frac{1}{2} \theta^2)+b^2_j}\right)t\\
&& \int_0^{t} \exp(-a_j \left[ \int_0^{t-z}  \sqrt{-2  g_j(u,s,\theta)+b^2_j}\;  ds +(b_j+\sqrt{-2  (\mu^j_1 \theta+\frac{1}{2} \theta^2)+b^2_j})z \right])  z^{m-1}e^{-(\lambda_{2}-\theta)z} \;dz
 \end{eqnarray*}
\end{example}

\section{Conclusions}
A switching mean-reverting model for temperatures when the latter oscillates between two stochastic differential equations with time-changed Levy noises has been investigated. The characteristic function and Esscher parameter to obtain an Equivalent Martingale Measure have been found via an approximate closed-form expression. Calculations require some rather complex but feasible numerical approximations involving numerical calculations of double integrals, truncated series and root solving.\\
This theoretical framework paves the way to price weather financial contracts based on temperature indices such as cumulative average temperatures, cooling and heating days.
\section{Acknowledgements}
This research has been supported by the  Natural Sciences and Engineering Research Council of Canada (NSERC) and the Environmental Finance Risk Management program, Institute of Environment, Florida International University, US.
\section{Appendix}
To compute the derivatives of the characteristic function we first find the derivatives of the intermediate functions $I^{\theta}_j(u,x), J^{\theta}_j(u), G^{\theta}_j(u,t,m), J^{\theta}_3(u,t,k, \theta)$, $J^{\theta}_4(u,t,k, \theta)$, $C_2(u, \theta,k) $ and $C_3(u, \theta, k)$.\\
 Since:
\begin{equation*}
  g_j(u \sigma^{-1}e^{\alpha t},s,\theta)=-iu \mu^j_1 e^{\alpha s} +  \mu^j_1 \theta + \frac{1}{2}(-i u e^{\alpha s}+\theta)^2,\;j=1,2.
\end{equation*}
we have:
\begin{eqnarray*}
\frac{d g_j(u \sigma^{-1}e^{\alpha t},s,\theta)}{du} &=&  -i \mu^j_1 e^{\alpha s}-i e^{\alpha s}  (-i u e^{\alpha s}+\theta)\\
&=&  -i \mu^j_1 e^{\alpha s}- u e^{2 \alpha s}-i \theta e^{\alpha s}=-i(\mu^j_1 - i u e^{ \alpha s}+ \theta) e^{\alpha s}
\end{eqnarray*}
Then:
\begin{eqnarray*}
 \frac{d I^{\theta}_j(u  \sigma^{-1}e^{\alpha t},x)}{du}|_{u=0}&=&  \int_0^x \frac{d l_{R^{j}}(g_j(u \sigma^{-1}e^{\alpha t},s,\theta))}{du}|_{u=0} \frac{dg_j(u \sigma^{-1}e^{\alpha t},s,\theta)}{du}|_{u=0} \;  ds \\
 &=&  -i (\mu^j_1+ \theta) \frac{d l_{R^{j}}}{du}(\mu^j_1 \theta+ \frac{1}{2}\theta^2) \int_0^x e^{\alpha s} \;  ds \\
 &=& -i \frac{(\mu^j_1+ \theta)}{\alpha} \frac{d l_{R^{j}}}{du}(\mu^j_1 \theta+ \frac{1}{2}\theta^2)(e^{\alpha x}-1)=-i L_j(\theta, \alpha)(e^{\alpha x}-1)
\end{eqnarray*}
Hence, taking into account:
\begin{eqnarray*}
I^{\theta}_{l}(0,x) &=&  \int_0^x l_{R^{j}}(  \mu^j_1 \theta +\frac{1}{2} \theta^2)\;  ds-l_{R^j}(\mu^j_1 \theta+\frac{1}{2} \theta^2))x=0
\end{eqnarray*}
we have:
\begin{eqnarray*}
&& \frac{d G^{\theta}_j(u \sigma^{-1}e^{\alpha t},t,m)}{du} =  \int_0^{t} \frac{d I^{\theta}_j(u \sigma^{-1}e^{\alpha t},t-z)}{du} e^{I^{\theta}_j(u \sigma^{-1}e^{\alpha t},t-z)}  z^{m-1}e^{-(\lambda_{21}-\theta)z} \;dz \\
&& \frac{d G^{\theta}_j(u \sigma^{-1}e^{\alpha t},t,m)}{du}|_{u=0} = -i L_1(\theta, \alpha) \int_0^{t} e^{I^{\theta}_j(0,t-z)} (e^{\alpha(t-z)}-1)  z^{m-1}e^{-(\lambda_{21}-\theta)z} \;dz \\
&=& -i L_j(\theta, \alpha)  \left[\int_0^{t}  z^{m-1}e^{-(\lambda_{21}-\theta)z+\alpha(t-z)} \;dz-\int_0^{t}   z^{m-1}e^{-(\lambda_{21}-\theta)z} \;dz \right]\\
&=& -i L_j(\theta, \alpha)  \left[e^{\alpha t} \int_0^{t}   z^{m-1}e^{-(\lambda_{21}-\theta+\alpha)z} \;dz - \int_0^{t} z^{m-1}e^{-(\lambda_{21}-\theta)z} \;dz \right]\\
&=& -i L_j(\theta, \alpha)  \left[e^{\alpha t}(\alpha+\lambda_{21}-\theta)^{-m} \gamma(m,\alpha+\lambda_{21}-\theta)-(\lambda_{21}-\theta)^{-m} \gamma(m,\lambda_{21}-\theta)\right]
\end{eqnarray*}
Moreover:
\begin{eqnarray}\nonumber
\frac{d J^{\theta}_1(u \sigma^{-1}e^{\alpha t})}{du} &=& \int_0^{+\infty} \frac{dI^{\theta}_{1}(u \sigma^{-1}e^{\alpha t},x)}{du}e^{I^{\theta}_{l}(u \sigma^{-1}e^{\alpha t},x)}e^{-(\lambda_{12}-\theta)x}\; dx \\ \nonumber
\frac{d J^{\theta}_1(u \sigma^{-1}e^{\alpha t})}{du}|_{u=0} &=& \int_0^{+\infty} \frac{dI^{\theta}_{1}(u \sigma^{-1}e^{\alpha t},x)}{du}|_{u=0} e^{I^{\theta}_{l}(0,x)}e^{-(\lambda_{12}-\theta)x}\; dx\\ \nonumber
&=& -i L_1(\theta, \alpha) \int_0^{+\infty} (e^{\alpha x}-1) e^{-(\lambda_{12}-\theta)x}\; dx\\ \nonumber
&=&- i L_1(\theta, \alpha) \int_0^{+\infty} (e^{\alpha x}-1) e^{-(\lambda_{12}-\theta)x}\; dx\\ \nonumber
&=&- i L_1(\theta, \alpha)  \left[\int_0^{+\infty} e^{\alpha x} e^{-(\lambda_{l}-\theta)x}\; dx-\int_0^{+\infty}  e^{-(\lambda_{12}-\theta)x}\; dx \right]\\ \nonumber
&=&- i L_1(\theta, \alpha)  \left[\frac{1}{\lambda_{12}-\alpha-\theta}-\frac{1}{\lambda_{12}-\theta} \right]\\ \nonumber
&=&    - \frac{i \alpha L_1(\theta, \alpha)}{(\lambda_{12}-\theta)(\lambda_{12}-\alpha -\theta)}\\ \label{eq:jk}
 \end{eqnarray}
 under the condition $ \lambda_{12}> \theta+\alpha$.\\
 Similarly
 \begin{eqnarray}\nonumber
\frac{d J^{\theta}_2(u \sigma^{-1}e^{\alpha t})}{du} &=&  - \frac{i \alpha L_2(\theta, \alpha)}{(\lambda_{21}-\theta)(\lambda_{21}-\alpha -\theta)}
 \end{eqnarray}
 On the other hand:
\begin{eqnarray*}
\frac{d [J^{\theta}_1(u \sigma^{-1}e^{\alpha t})J^{\theta}_2(u \sigma^{-1}e^{\alpha t})]^k}{du} &=&  \frac{d [J^{\theta}_1(u \sigma^{-1}e^{\alpha t})]^k}{du}[J^{\theta}_2(u \sigma^{-1}e^{\alpha t})]^k \\
&+& \frac{d [J^{\theta}_2(u \sigma^{-1}e^{\alpha t})]^k}{du} [J^{\theta}_1(u \sigma^{-1}e^{\alpha t})]^k \\
&=& k [J^{\theta}_1(u \sigma^{-1}e^{\alpha t})]^{k-1}[J^{\theta}_2(u \sigma^{-1}e^{\alpha t})]^k \frac{d J^{\theta}_1(u \sigma^{-1}e^{\alpha t})}{du}\\
&+& k [J^{\theta}_2(u \sigma^{-1}e^{\alpha t})]^{k-1} [J^{\theta}_1(u \sigma^{-1}e^{\alpha t}\sigma^{-1}e^{\alpha t})]^k \frac{d J^{\theta}_2(u \sigma^{-1}e^{\alpha t})}{du}
\end{eqnarray*}
Hence:
\begin{eqnarray}\nonumber
&&\frac{d (J^{\theta}_1(u \sigma^{-1}e^{\alpha t}))^k(J^{\theta}_2(u \sigma^{-1}e^{\alpha t}))^k}{du}|_{u=0} = k [J^{\theta}_1(0)]^{k-1}[J^{\theta}_2(0)]^k \frac{d J^{\theta}_1(u \sigma^{-1}e^{\alpha t})}{du}|_{u=0}\\ \nonumber
&+& k [J^{\theta}_2(0)]^{k-1}[J^{\theta}_1(0)]^k \frac{d J^{\theta}_2(u \sigma^{-1}e^{\alpha t})}{du}|_{u=0}\\ \nonumber
&=& k [J^{\theta}_1(0)]^{k-1}[J^{\theta}_2(0)]^{k-1} \left[\frac{d J^{\theta}_1(u \sigma^{-1}e^{\alpha t})}{du}|_{u=0}J^{\theta}_2(0) + \frac{d J^{\theta}_2(u \sigma^{-1}e^{\alpha t})}{du}|_{u=0}J^{\theta}_1(0) \right] \\ \label{eq:jaysprod}
&&
\end{eqnarray}
But:
\begin{eqnarray*}
 J^{\theta}_1(0)  &=& \int_0^{+\infty} e^{-(\lambda_{12}-\theta)x}\; dx =\frac{1}{\lambda_{12}-\theta}\\
 J^{\theta}_2(0)  &=& \int_0^{+\infty} e^{-(\lambda_{21}-\theta)x}\; dx =\frac{1}{\lambda_{21}-\theta}
\end{eqnarray*}
Then, we have after substituting into (\ref{eq:jaysprod}) that:
\begin{eqnarray}\nonumber
\frac{d (J^{\theta}_1(u))^k(J^{\theta}_2(u))^k}{du}|_{u=0} &=& \frac{k}{(\lambda_{12}-\theta)^{k-1}(\lambda_{21}-\theta)^{k-1}}\\ \nonumber
&& \left[\frac{-i \alpha L_1(\theta, \alpha)}{(\lambda_{12}-\theta)(\lambda_{12}-\alpha -\theta)(\lambda_{21}-\theta)} - \frac{i \alpha L_2(\theta,\alpha)}{(\lambda_{21}-\theta)(\lambda_{21}-\alpha -\theta)(\lambda_{12}-\theta)} \right]\\ \nonumber
 &=& -i \frac{k \alpha}{(\lambda_{12}-\theta)^{k}(\lambda_{21}-\theta)^{k}} \left[\frac{  L_1(\theta, \alpha)}{(\lambda_{12}-\alpha -\theta)} + \frac{L_2(\theta, \alpha)}{(\lambda_{21}-\alpha -\theta)} \right]\\ \label{eq:jlk}
 &&
\end{eqnarray}
Moreover:
\begin{eqnarray}\nonumber
 \frac{d J_3(u\sigma^{-1}e^{\alpha t},k, \theta)}{du}|_{u=0} &=& D(k, -k, 2k, \theta) \sum_{l=0}^{+\infty}C(k,2k,l) \frac{dG^{\theta}_1(u\sigma^{-1}e^{\alpha t},t,2k+l)}{du}|_{u=0}\lambda^l_{21} \\ \nonumber
 \frac{d J_4(u \sigma^{-1}e^{\alpha t},k, \theta)}{du}|_{u=0} &=& D(k+1, -k, 2k+1,\theta)\\ \nonumber
 && \sum_{l=0}^{+\infty}C(k+1,2k+1,l)\frac{d G^{\theta}_2(u\sigma^{-1}e^{\alpha t},t,2k+l+1)}{du}|_{u=0} \lambda^l_{21} \\ \label{eq:jthetak}
 &&
\end{eqnarray}
Then, taking into account equations (\ref{eq:jk}) and (\ref{eq:jlk})  we have:
\begin{eqnarray*}
\frac{dC_2(u \sigma^{-1}e^{\alpha t}, \theta,k)}{du} &=&  (\lambda_{12}-\theta)^k (\lambda_{21}-\theta)^k \frac{d[J^{\theta}_1(u \sigma^{-1}e^{\alpha t})]^k [J^{\theta}_2(u \sigma^{-1}e^{\alpha t})]^k}{du} \\
 \frac{dC_2(u \sigma^{-1}e^{\alpha t}, \theta,k)}{du}|_{u=0} &=& -i k \alpha \left[\frac{  L_1(\theta, \alpha)}{(\lambda_{12}-\alpha -\theta)} + \frac{L_2(\theta, \alpha)}{(\lambda_{21}-\alpha -\theta)} \right]
\end{eqnarray*}
and
\begin{eqnarray*}
\frac{dC_3(u \sigma^{-1}e^{\alpha t}, \theta,k)}{du} &=& (\lambda_{12}-\theta)\frac{dJ^{\theta}_1(u)}{du} p^{\theta}_{2k+1}(t)\\
\frac{dC_3(u \sigma^{-1}e^{\alpha t}, \theta,k)}{du}|_{u=0} &=&  (\lambda_{12}-\theta)\frac{dJ^{\theta}_1(u \sigma^{-1}e^{\alpha t})}{du}|_{u=0} p^{\theta}_{2k+1}(t)\\
&=&  - i \frac{\alpha L_1(\theta, \alpha)}{(\lambda_{12}-\alpha -\theta)}   p^{\theta}_{2k+1}(t)\\
\end{eqnarray*}
Then,
 \begin{eqnarray} \nonumber
 \frac{d C_4(u \sigma^{-1}e^{\alpha t}, \theta)}{du} &=& \sum_{k=0}^{+\infty}  \left[ \frac{ dC_2(u \sigma^{-1}e^{\alpha t}, \theta,k)}{du} C_3(u \sigma^{-1}e^{\alpha t}, \theta, k) \right. \\  \nonumber
 &+& \left. \frac{dC_3(u \sigma^{-1}e^{\alpha t}, \theta,k)}{du} C_2(u \sigma^{-1}e^{\alpha t}, \theta, k) \right]\\ \nonumber
 D_1(t, \theta) &=& \frac{dC_4(u \sigma^{-1}e^{\alpha t}, \theta)}{du}|_{u=0}= \sum_{k=0}^{+\infty}  \left[ \frac{dC_2(u \sigma^{-1}e^{\alpha t}, \theta,k)}{du}|_{u=0} C_3(0, \theta, k) \right. \\ \nonumber
 &+& \left. \frac{dC_3(u \sigma^{-1}e^{\alpha t}, \theta,k)}{du}|_{u=0}  \right]\\ \nonumber
 &=& \sum_{k=0}^{+\infty}  \left[  -i k \alpha \left(\frac{  L_1(\theta, \alpha)}{(\lambda_{12}-\alpha -\theta)} + \frac{L_2(\theta, \alpha)}{(\lambda_{21}-\alpha -\theta)} \right) (p^{\theta}_{2k}(t)+p^{\theta}_{2k+1}(t)) \right. \\ \nonumber
 &-& \left.  i \alpha \frac{ L_1(\theta, \alpha)}{(\lambda_{12}-\alpha -\theta)}   p^{\theta}_{2k+1}(t) \right]\\ \nonumber
 &=&  -i \alpha \left[ \sum_{k=0}^{+\infty}   k  \left(\frac{  L_1(\theta, \alpha)}{(\lambda_{12}-\alpha -\theta)} + \frac{L_2(\theta, \alpha)}{(\lambda_{21}-\alpha -\theta)} \right) (p^{\theta}_{2k}(t)+p^{\theta}_{2k+1}(t)) \right. \\ \nonumber
 &+& \left. ( \frac{ L_1(\theta, \alpha)}{(\lambda_{12}-\alpha -\theta)}   p^{\theta}_{2k+1}(t)) \right] \\ \nonumber
 &=&  -i \alpha \left[\left(\frac{  L_1(\theta, \alpha)}{(\lambda_{12}-\alpha -\theta)} + \frac{L_2(\theta, \alpha)}{(\lambda_{21}-\alpha -\theta)} \right) \sum_{k=0}^{+\infty}   k  p^{\theta}_{k}(t) \right. \\ \nonumber
 &+& \left.  \frac{ L_1(\theta, \alpha)}{(\lambda_{12}-\alpha -\theta)}  \sum_{k=0}^{+\infty} p^{\theta}_{2k+1}(t) \right] \\ \label{eq:D10}
 &&
  \end{eqnarray}
  from which equation (\ref{eq:D1F}) follows.\\
 Moreover,
 \begin{eqnarray}\nonumber
 \frac{dC_5(u \sigma^{-1}e^{\alpha t}, \theta)}{du} &=& \sum_{k=0}^{+\infty}   \frac{dJ^{\theta}_3(u \sigma^{-1}e^{\alpha t},k)}{du}  p^{\theta}_{2k}(t) \\ \nonumber
 &+&  \sum_{k=0}^{+\infty}   \frac{dJ^{\theta}_4(u \sigma^{-1}e^{\alpha t},k)}{du}p^{\theta}_{2k+1}(t)
  \end{eqnarray}
\bibliography{ReferenceswitchingWD}
\bibliographystyle{plain}
 \end{document}